\documentclass{amsart}
\usepackage{amsfonts}
\usepackage{amsmath,amssymb}
\usepackage{amsthm}
\usepackage{amscd}
\usepackage{graphics}
\usepackage{graphicx}

\theoremstyle{remark}{
\newtheorem{Def}{{\rm Definition}}
\newtheorem{Ex}{{\rm Example}}

\newtheorem{Prob}{{\rm Problem}}

}
\theoremstyle{plain}
{
\newtheorem{Cor}{Corollary}
\newtheorem{Prop}{Proposition}
\newtheorem{Thm}{Theorem}

}

\begin{document}
\title[Arrangements of circles supported by small chords.]{Arrangements of circles supported by small chords and compatible with natural real algebraic functions}
\author{Naoki kitazawa}

\address{Institute of Mathematics for Industry, Kyushu University, 744 Motooka, Nishi-ku Fukuoka 819-0395, Japan\\
 TEL (Office): +81-92-802-4402 \\
 FAX (Office): +81-92-802-4405 \\
}
\email{n-kitazawa@imi.kyushu-u.ac.jp, naokikitazawa.formath@gmail.com}
\urladdr{https://naokikitazawa.github.io/NaokiKitazawa.html}
\maketitle
\begin{abstract}
We have previously proposed a study of arrangements of small circles which also surround regions in the plane realized as the images of natural real algebraic maps yielding {\it Morse-Bott} functions by projections. Among studies of arrangements, families of smooth regular submanifolds in smooth manifolds, this study is fundamental, explicit, and new, surprisingly.

We have obtained a complete list of local changes of the graphs the regions naturally collapse to in adding a (generic) small circle to an existing arrangement of the proposed class. Here, we propose a similar and essentially different class of arrangements of circles. The present study also yields real algebraic maps and nice real algebraic functions similarly and we present a similar study.

We are interested in topological properties and combinatorics among such arrangements and regions and applications to constructing such real algebraic maps and manifolds explicitly and understanding their global structures.

\end{abstract}
\section{Introduction.}
\label{sec:1}

Arrangements of linear (affine) subspaces, circles, and, more generally, smooth regular submanifolds, in Euclidean spaces, have produced various studies in various mathematics and we cannot explain systematically or explicitly related studies well. \cite{tamaki2, tamaki3} are from webpages \cite{tamaki1}, explaining various topics on mathematics: most of the exposition is presented respecting algebraic topology, mainly. Dai Tamaki is responsible for them mainly, and various mathematicians support. \cite{tamaki2, tamaki3} are for arrangements. Our study is concerned with circles. As a related study \cite{carmesinschulz} ({\cite[Section 1 Related work]{carmesinschulz}}) implies, it is also difficult to explain related studies in the case of circles.

Including the presentation ever, presentations in the present paper respect \cite{kitazawa3} considerably. 
We refer to \cite{kitazawa3} as a closely related study in various scenes where we do not assume mathematical arguments there. 

Previously, in \cite{kitazawa3}, we have proposed several 
explicit classes of arrangements of small circles, surrounding regions in the plane realized as the images of natural real algebraic maps yielding {\it Morse-Bott} functions by composing projections. As our main result there, we have classified local changes of the graphs the regions naturally collapse to in adding a (generic) small circle to an existing arrangement of the proposed class. We propose similar and new classes in our present paper.

Originally, we have been interested in construction of such real algebraic functions, maps and manifolds explicitly and understanding their global structures. Such interest comes from, for example, geometric theory of Morse functions, explicit construction of functions, maps and manifolds in singularity theory of smooth maps, and construction in real algebraic geometry. We do not assume advanced knowledge or experience on related geometry and singularity theory. We only introduce related books and articles. We concentrate on arrangements of circles and regions surrounded by the circles. We present important classes of arrangements of circles. Before this, we explain fundamental notation we need.
\subsection{Notation on topological spaces, manifolds and maps.}
We present fundamental notation. Let ${\mathbb{R}}^n$ denote the $n$-dimensional Euclidean space, which is a smooth manifold equipped with the standard Euclidean metric, and $||x|| \geq 0$ the distance between $x \in {\mathbb{R}}^n$ and the origin $0 \in {\mathbb{R}}^n$. We use $\mathbb{R}:={\mathbb{R}}^1$. Let ${\pi}_{m,n,i}:{\mathbb{R}}^m \rightarrow {\mathbb{R}}^n$ with $m>n \geq 1$ denote the canonical projection: let $x=(x_1,x_2) \in {\mathbb{R}}^{n} \times {\mathbb{R}}^{m-n}={\mathbb{R}}^m$ and ${\pi}_{m,n,i}(x)=x_i$ ($i=1,2$). Let $D^k:=\{x \in {\mathbb{R}}^k \mid ||x|| \leq 1\}$ denote the $k$-dimensional unit disk and $S^k:=\{x \in {\mathbb{R}}^{k+1} \mid ||x||=1\}$ the $k$-dimensional unit sphere. For a subspace $Y \subset X$ of a topological space $X$, let $\overline{Y}$ denote the closure and $Y^{\circ}$ the interior of $Y$. For $Y, \overline{Y}, Y^{\circ} \subset X$, we do not refer to the outer space $X$ unless otherwise stated. For a (topological) manifold $X$ whose boundary is non-empty, let $\partial X$ denote the boundary.  
\subsection{Our arrangements of circles.}

\label{subsec:1.2}
We first review some from \cite{kitazawa2, kitazawa3}.

\begin{Def}
\label{def:1}
For a circle $S _{x_0,r}:=\{x \mid ||x-x_0||=r\} \subset {\mathbb{R}}^2$ ($x_0=(x_{0,1},x_{0,2}) \in {\mathbb{R}}^2$, $r>0$), the points $(x_{0,1} \pm r,x_{0,2})$ {\rm (}$(x_{0,1},x_{0,2} \pm r)${\rm )} are {\it vertical} {\rm (}resp. {\it horizontal}{\rm )} poles.
\end{Def}
\begin{Def}
\label{def:2}
A pair of a family $\mathcal{S}=\{S _{x_j,r_j}:=\{x_j \mid ||x_j-x_{j,0}||=r_j\} \subset {\mathbb{R}}^2\}$ of circles satisfying the following and a region $D_{\mathcal{S}} \subset {\mathbb{R}}^2$ is called an {\it MBC arrangement}. 
\begin{enumerate}
\item The  region $D_{\mathcal{S}} \subset {\mathbb{R}}^2$ is a bounded connected component of ${\mathbb{R}}^2-{\bigcup}_{S _{x_j,r_j} \in \mathcal{S}} S _{x_j,r_j}$.
\item The intersection $\overline{D_{\mathcal{S}}} \bigcap S _{x_j,r_j}$ is not empty for each circle $S _{x_j,r_j} \in \mathcal{S}$.
\item At points in $\overline{D_{\mathcal{S}}}$, at most two distinct circles $S_{x_{i_1},r_{i_1}}, S_{x_{i_2},r_{i_2}} \in \mathcal{S}$ intersect and the intersection satisfies the following: for each point $p_{i_1,i_2}$ there, the sum of the tangent vector spaces of $S_{x_{i_1},r_{i_1}}$  and $S_{x_{i_2},r_{i_2}}$ at $p_{i_1,i_2}$ coincides with the tangent vector space of ${\mathbb{R}}^2$ at $p_{i_1,i_2}$. Furthermore, these points are not vertical poles or horizontal poles of the circles from $\mathcal{S}$.

\end{enumerate}
\end{Def}
The following show new classes of arrangements of circles.
\begin{Def}
\label{def:3}
A pair of a set $\mathcal{S}:=\{S_{x_{j},r_{j}}\}$ of circles of fixed radii in ${\mathbb{R}}^2$ and a region $D_{\mathcal{S}} \subset {\mathbb{R}}^2$ which is also one of a bounded connected component of the complementary set ${\mathbb{R}}^2-{\bigcup}_{S_{x_{j},r_{j}} \in \mathcal{S}} S_{x_{j},r_{j}}$ is called a {\it normally inductive arrangement} ({\it NI arrangement}) {\it of circles} if we can have $\mathcal{S}$ inductively in the following way.
\begin{enumerate}
\item \label{def:3.1} Choose a non-empty set ${\mathcal{S}}_0$ of mutually disjoint circles $S_{x_{j,0},r_{j,0}} \in {\mathcal{S}}_0$ of fixed radii in ${\mathbb{R}}^2$ such that the boundary $\partial \overline{D_{{\mathcal{S}}_0}}$ of the closure $\overline{D_{{\mathcal{S}}_0}}$ of a bounded connected component $D_{{\mathcal{S}}_0} \subset {\mathbb{R}}^2$ of the complementary set ${\mathbb{R}}^2-{\bigcup}_{S_{x_{j,0},r_{j,0}} \in 
{\mathcal{S}}_0} S_{x_{j,0},r_{j,0}}$ is the disjoint union of all circles $S_{x_{j,0},r_{j,0}}$.
\item \label{def:3.2} Let ${\mathcal{S}}:={\mathcal{S}}_0$. We consider the following procedure inductively.
First we choose a new circle $S_{x_{j^{\prime}},r_{j^{\prime}}}$ of a fixed radius $r_{j^{\prime}}>0$ centered at $x_{j^{\prime}} \in {\mathbb{R}}^2$ and intersecting at least one circle from the existing set ${\mathcal{S}}$. We define the new set ${\mathcal{S}}^{\prime}$ by adding the new circle there. Two distinct circles $S_{x_{j},r_{j}}$ and $S_{x_{j^{\prime}},r_{j^{\prime}}}$ always intersect according to the rule as in Definition \ref{def:2} at points in $\overline{D_{\mathcal{S}}}$: for each point $p_{j,j^{\prime}}$ of such points, the sum of the tangent vector spaces of $S_{x_{j},r_{j}}$  and $S_{x_{j^{\prime}},r_{j^{\prime}}}$ at $p_{j,j^{\prime}}$ coincides with the tangent vector space of ${\mathbb{R}}^2$ at $p_{j,j^{\prime}}$ and three distinct circles $S_{x_{j_1},r_{j_1}}$, $S_{x_{j_2},r_{j_2}}$, and $S_{x_{j^{\prime}},r_{j^{\prime}}}$ do not intersect at any point in $\overline{D_{\mathcal{S}}}$. We also define the new region $D_{{\mathcal{S}}^{\prime}}$ as the intersection $D_{\mathcal{S}} \bigcap {E_{x_{j^{\prime}},r_{j^{\prime}}}}^{\circ}$ where $E_{x_{j^{\prime}},r_{j^{\prime}}}:={D_{x_{j^{\prime}},r_{j^{\prime}}}}$ for the closed disk $D_{x_{j^{\prime}},r_{j^{\prime}}}$ whose boundary is $S_{x_{j^{\prime}},r_{j^{\prime}}}$ in ${\mathbb{R}}^2$ or $E_{x_{j^{\prime}},r_{j^{\prime}}}:={\mathbb{R}}^2-D_{x_{j^{\prime}},r_{j^{\prime}}}$. We pose a constraint that the intersections $\overline{D_{{\mathcal{S}}}^{\prime}} \bigcap S _{x_j,r_j}$ and $\overline{D_{{\mathcal{S}}}^{\prime}} \bigcap S _{x_{j^{\prime}},r_{j^{\prime}}}$ are non-empty for all considered circles $S _{x_j,r_j}, S_{x_{j^{\prime}},r_{j^{\prime}}} \in {\mathcal{S}}^{\prime}$. We redefine $(\mathcal{S},D_{\mathcal{S}}):=({\mathcal{S}}^{\prime},D_{{\mathcal{S}}^{\prime}})$.
\end{enumerate}
\end{Def}
\begin{Def}
\label{def:4} 
In Definition \ref{def:3}, we define classes of NI arrangements of circles according to additional rules in the step {\rm (}\ref{def:3.2}{\rm )}.
\begin{itemize}
\item If each point $p_{j,j^{\prime}}$ of the intersections of the circles is not a vertical pole or a horizontal pole of any of the circles, then by our definition the resulting pair of the family of circles and the region is an MBC arrangement and defined to be a {\it normally inductive} MBC arrangement (an {\it NI-MBC arrangement}).  
\item If the set ${\bigcup}_{S _{x_j,r_j} \in \mathcal{S}} S_{x_j,r_j} \bigcap D_{x_{j^{\prime}},r_{j^{\prime}}}$ is connected and non-empty for the circle $S _{x_{j^{\prime}},r_{j^{\prime}}}$ at each step, then we call the resulting pair of the family of circles and the region a {\it normally connected inductive} arrangement of circles (an {\it NCI arrangement of circles}).
\item We call an NCI arrangement of circles which is also an NI-MBC arrangement is an {\it NCI-MBC arrangement}.
\item {\rm (}\cite{kitazawa3}{\rm )} If each circle $S _{x_{j^{\prime}},r_{j^{\prime}}}$ is centered at a point of a circle of $\mathcal{S}$ and sufficiently small and $E_{x_{j^{\prime}},r_{j^{\prime}}}:={\mathbb{R}}^2-D_{x_{j^{\prime}},r_{j^{\prime}}}$, then it is an MBC arrangement {\rm (}{\it MBCC arrangement}{\rm )}.
\end{itemize}  
\end{Def}

\subsection{Organization of our paper and our main work.}
In the next section, we define {\it Poincar\'e-Reeb graphs} of the regions ${\mathcal{D}}_{\mathcal{S}} \subset {\mathbb{R}}^2$, which are the spaces of all connected components of all preimages of single points of the restrictions of the projections ${\pi}_{2,1,i}$ to the closures of the regions, endowed with the quotient topologies. The vertex set of such a graph is the set of all connected components containing horizontal poles, vertical poles, or points contained in exactly two circles of $\mathcal{S}$. They respect \cite{bodinpopescupampusorea, kohnpieneranestadrydellshapirosinnsoreatelen, sorea1, sorea2} for example, and are formulated rigorously first in \cite{kitazawa3}.  The regions collapse to the graphs naturally. We explain that the closures of our regions are seen as the images of natural real algebraic maps locally like so-called moment-maps and generalizing the canonical projection ${\pi}_{m+1,2} {\mid}_{S^m}$ of the unit sphere $S^m$ ($m \geq 2$), the case $\mathcal{S}=\{S^1\}$ with $D_{\mathcal{S}}={(D^2)}^{\circ}$. By composing the projections, we have nice smooth functions. We do not need to understand such maps or related singularity theory with geometry and we leave it to \cite{kitazawa1, kitazawa2}.

In the third section, we define a new class of NI arrangements of circles: NI arrangements {\it supported by small chords} ({\it SSC-NI arrangements}). In short, in Definition \ref{def:3}, at each step in (\ref{def:3.2}), we choose a new circle passing the two distinct sufficiently close points contained in some circles of $\mathcal{S}$ and having an arc sufficiently close to the chord connecting the previous two points, as a subspace. We investigate local changes of the regions (Poincar\'e-Reeb graphs). We have investigated the changes in the case of MBCC arrangements in \cite{kitazawa3}. We also compare the results.

\section{Poincar\'e-Reeb graphs of the region $D_{\mathcal{S}}$ and real algebraic maps onto (the closure) of the region.}
\label{sec:2}

\subsection{Dimensions of topological spaces decomposed into cells, smooth maps and graphs.}
\label{subsec:2.1}
For a non-empty topological space $X$ decomposed into cells, we can define the dimension $\dim X$ as the unique non-negative integer. This gives a topological invariant.
(Topological) manifolds, graphs, polyhedra, and CW complexes are of such a class.

For a smooth manifold $X$, let $T_p X$ denote the tangent vector space at $p \in X$.
For a smooth map $c:X \rightarrow Y$ between smooth manifolds $X$ and $Y$, the differential ${dc}_p:T_pX \rightarrow T_{c(p)} Y$ at $p \in X$ is a linear map between the tangent vector spaces. The point $p$ is a {\it singular} point of $c$ if the rank of the differential ${dc}_p$ drops. The space $C^{\infty}(X,Y)$ of all smooth maps between the manifolds can be topologized with the {\it Whitney $C^{\infty}$ topology}. We do not need to understand related definitions and theory rigorously. We only note that this topology is based on whether two smooth maps are close not only in the values of the maps at each point but also in the values of the derivations.   

A {\it diffeomorphism} is a smooth map which is a homeomorphism and which has no singular point. A diffeomorphism on a smooth manifold $X$ is a diffeomorphism from $X$ onto itself. The {\it diffeomorphism group} of $X$ is the group consisting of all diffeomorphisms on $X$.

We define a class of bundles: a {\it smooth} bundle. A {\it smooth} bundle is a bundle whose fiber is a smooth manifold and whose structure group is regarded as a subgroup of the diffeomorphism group of the fiber.

A graph is a CW complex consisting of $0$-cells ({\it vertices}) and $1$-cells ({\it edges}). The set of all vertices (edges) is the vertex (resp. {\it edge}) set of the graph.
An {\it oriented} graph or a {\it digraph} is a graph whose edges are all oriented. A {\it V-digraph} $K$ is a digraph associated with a map $l_K:V_K \rightarrow P$ from the vertex set $V_K$ into an totally ordered set $P$ and oriented according to the values $l(v)$: the relation $l_K(v_1) \neq l_K(v_2)$ must hold if there exists an edge $e$ connecting $v_1$ and $v_2$ and $e$ is oriented as an edge departing from $v_1$ and entering $v_2$ if $l_K(v_1)<l_K(v_2)$ where $<$ denotes the order on $P$.
\begin{Def}
Two graphs $K_1$ and $K_2$ are {\it isomorphic} if there exists a piecewise smooth homeomorphism $\phi:K_1 \rightarrow K_2$ mapping the vertex set of $K_1$ onto that of $K_2$: this is an {\it isomorphism} between the graphs.
Two digraphs are {\it isomorphic} if in addition there exists an isomorphism of the graphs preserving the orientations of the edges: this is an {\it isomorphism} between the digraphs.
Two V-digraphs are {\it isomorphic} if in addition there exists an isomorphism of the digraphs preserving the orders of the values of $l_K$: this is an {\it isomorphism} between the V-digraphs.
\end{Def}

\subsection{A Poincar\'e-Reeb graph, our region surrounded by the circles collapses.}
\label{subsec:2.2}
We review a graph for $(\mathcal{S},D_{\mathcal{S}})$. This is the set of all connected components of all preimages of all single points for the restriction ${\pi}_{2,1,i} {\mid}_{\overline{D_{\mathcal{S}}}}$ of the projection. We can topologize the set with the quotient topology. A {\it Poincar\'e-Reeb graph} of $D_{\mathcal{S}}$ for ${\pi}_{2,1,i} {\mid}_{\overline{D_{\mathcal{S}}}}$ is the graph whose vertex set consists of all points (connected components of the preimages) containing vertical poles (resp. horizontal poles) of some circles of $\mathcal{S}$ or points in exactly two distinct circles of $\mathcal{S}$ for $i=1$ (resp. $i=2$). This is shown to be a graph.
We review essential parts in showing that this is a graph where more precise exposition is given in \cite{kitazawa3}. We have given the exposition in \cite{kitazawa3} due to the situation that such exposition had not been presented ever. On interiors of edges, the projection ${\pi}_{2,1,i} {\mid}_{D_{\mathcal{S}}}$ gives a trivial smooth bundle whose fiber is diffeomorphic to a closed interval $D^1$, by so-called Ehresmann fibration theorem (or its relative version). 
We can easily check local structures around vertices. Let $q_{D_{\mathcal{S}},i}:\overline{D_{\mathcal{S}}} \rightarrow G_{D_{\mathcal{S}},i}$ denote the quotient map onto the graph $G_{D_{\mathcal{S}},i}$.

We can regard this as a V-digraph by the (restriction to the vertex set of the graph of the) uniquely defined function $V_{D_{\mathcal{S}},i}:G_{D_{\mathcal{S}},i} \rightarrow \mathbb{R}$ with the relation ${\pi}_{2,1,i} {\mid}_{\overline{D_{\mathcal{S}}}}=V_{D_{\mathcal{S}},i} \circ q_{D_{\mathcal{S}},i}$. Based on this, we call this V-graph the {\it Poincar\'e-Reeb V-digraph of $D_{\mathcal{S}}$ for ${\pi}_{2,1,i}$}.

As more general theory, \cite{saeki1, saeki2} explain this fact.

\subsection{A real algebraic map onto the closure $\overline{D_{\mathcal{S}}} \subset {\mathbb{R}}^2$.}
\label{subsec:2.3}
\cite{kitazawa2} presents a natural real algebraic map onto the closure $\overline{D_{\mathcal{S}}} \subset {\mathbb{R}}^2$ locally represented as moment maps. We explain this where important ingredients are left to \cite{kitazawa2}. We do not assume arguments there in our paper.

We can define a surjective map $m_{\mathcal{S},D_{\mathcal{S}}}:\mathcal{S} \rightarrow A$ onto some finite set $A$. We also define this according to the following rule: if two distinct circles $S_{x_{j_1},r_{j_1}}, S_{x_{j_2},r_{j_2}} \in \mathcal{S}$ intersect at some point in the closure $\overline{D_{\mathcal{S}}}$, then the values there are distinct. Let $m_{\mathcal{S},D_{\mathcal{S}},0}$ be a non-negative integer valued map on $A$.

Hereafter, we also see ${\mathbb{R}}^n$ as the vector space canonically. We also use the notation for vectors. We define $p-q \in {\mathbb{R}}^n$ canonically for example and $||p-q||$ is also defined of course.
Let $f_j(x)=||x-x_j||^2-{r_j}^2$.

Let $M_{\mathcal{S},D_{\mathcal{S},m_{\mathcal{S},D_{\mathcal{S}}},m_{\mathcal{S},D_{\mathcal{S}},0}}}:=\{(x,y) \in {\mathbb{R}}^2 \times {\prod}_{a \in A} {\mathbb{R}}^{m_{\mathcal{S},D_{\mathcal{S}},0}(a)+1} \mid {\prod}_{j \in {m_{\mathcal{S},D_{\mathcal{S}}}}^{-1}(a)} |f_j(x)|-||y_a||^2\}$: the vector $y_a$ shows the component of the vector $y$ (the ordered sequence $y$ of vectors) which corresponds to $a \in A$. This set $M_{\mathcal{S},D_{\mathcal{S},m_{\mathcal{S},D_{\mathcal{S}}},m_{\mathcal{S},D_{\mathcal{S}},0}}}$ is the zero set of the real polynomial map. This is also {\it non-singular}: a {\it non-singular} union of connected components of the zero set of a real polynomial map is defined by the implicit function theorem or the rank of the polynomial map canonically.

For real algebraic geometry, see \cite{bochnakcosteroy, kollar} for example. We do not need related knowledge of course.

We define the restriction $f_{\mathcal{S},D_{\mathcal{S},m_{\mathcal{S},D_{\mathcal{S}}},m_{\mathcal{S},D_{\mathcal{S}},0}}}:={\pi_{{\Sigma}_{a \in A} (m_{\mathcal{S},D_{\mathcal{S}},0}(a)+1)+2,2}} {\mid}_{M_{\mathcal{S},D_{\mathcal{S},m_{\mathcal{S},D_{\mathcal{S}}},m_{\mathcal{S},D_{\mathcal{S}},0}}}}:M_{\mathcal{S},D_{\mathcal{S},m_{\mathcal{S},D_{\mathcal{S}}},m_{\mathcal{S},D_{\mathcal{S}},0}}} \rightarrow {\mathbb{R}}^2$.

The Poincar\'e-Reeb V-digraph of $D_{\mathcal{S}}$ for ${\pi}_{2,1,i}$ is also seen as the {\it Reeb graph} of the function ${\pi}_{2,1,i} \circ f_{\mathcal{S},D_{\mathcal{S},m_{\mathcal{S},D_{\mathcal{S}}},m_{\mathcal{S},D_{\mathcal{S}},0}}}:M_{\mathcal{S},D_{\mathcal{S},m_{\mathcal{S},D_{\mathcal{S}}},m_{\mathcal{S},D_{\mathcal{S}},0}}} \rightarrow \mathbb{R}$. For an MBC arrangement, we have a so-called {\it Morse-Bott} function. 

The {\it Reeb graph} $W_c$ of a smooth function $c$ is a graph whose underlying set is the set of all connected components of all preimages of all single points for the smooth function $c$ and topologized in the way similar to the case of Poincar\'e-Reeb graphs. \cite{reeb} is a related pioneering study. 
By defining the vertex set as the set of all points (connected components of the preimages) containing some singular points of the function, this is a graph in considerable cases including our cases for example. \cite{saeki1, saeki2} explain this fact.
Reeb graphs have been important topological and combinatorial objects and strong tools in theory of Morse functions, various singularity theory of differentiable maps and applied mathematics such as visualizations. 
We can argue the structures of the graphs as in Poincar\'e-Reeb graphs (V-digraphs).
We can define the V-digraph $W_c$ as the {\it Reeb V-digraph} of $c$.

As a fundamental exercise on singularity theory, we can check that the Poincar\'e-Reeb V-digraph $G_{D_{\mathcal{S}},i}$ of $D_{\mathcal{S}}$ for ${\pi}_{2,1,i}$ and the V-digraph $W_c$ are isomorphic.
We do not need to understand singularity theory of differentiable (smooth) maps. For related study, see \cite{golubitskyguillemin} for example.

\section{Our new class of arrangements of circles: SSC-NI arrangements.}
\label{sec:3}
This section is a main ingredient of our paper and presents new observations, arguments and results.

Hereafter, for example, we assume elementary plane geometry. Especially, we assume fundamental knowledge and facts on segments, straight lines and circles in the Euclidean plane ${\mathbb{R}}^2$.

The {\it volume} of a subset of the Euclidean space is defined naturally from the standard Euclidean metric (if it is a Lebesgue measurable set for example). 

In the $1$-dimensional case, the volume is the length and in the $2$-dimensional case, the volume is the area, for example.
We only need very elementary knowledge on volumes and measures.

For an NI arrangement $\mathcal{S}$ of circles, a {\it chord of $\mathcal{S}$} means a segment $\{p_1+t(p_2-p_1) \mid 0 \leq t \leq 1\}$ connecting two distinct points $p_1, p_2 \in {\mathbb{R}}^2$ which intersect the union ${\bigcup}_{S_{x_j,r_j} \in \mathcal{S}} S_{x_j,r_j}$ in the two-point set $\{p_1,p_2\}$: we also pose a condition that each $p_i$ ($i=1,2$) is in a single circle $S_{x_{i_j},r_{i_j}} \in \mathcal{S}$, and that for such a point $p_i$ ($i=1,2$), the sum of the tangent vector space of the chord of $\mathcal{S}$ at $p_i$ and the tangent vector space of $S_{x_{i_j},r_{i_j}}$ at $p_i$ is the tangent vector space of ${\mathbb{R}}^2$ there, and call these two points {\it chord boundary points} of the chord of $\mathcal{S}$. A {\it straight secant of $\mathcal{S}$} means a straight line $\{p_1+t(p_2-p_1) \mid t \in \mathbb{R}\}$ containing the chord $\{p_1+t(p_2-p_1) \mid 0 \leq t \leq 1\}$
of $\mathcal{S}$ as a subset. A circle containing the two chord boundary points of a chord of $\mathcal{S}$ is called a {\it circular secant of $\mathcal{S}$}: we also pose a condition that for such a point $p_i$ ($i=1,2$), the sum of the tangent vector space of the circular secant of $\mathcal{S}$ at $p_i$ and the tangent vector space of $S_{x_{i_j},r_{i_j}}$ at $p_i$ is the tangent vector space of ${\mathbb{R}}^2$ there.
We also call both a straight secant and a circular secant of $\mathcal{S}$ a {\it secant of $\mathcal{S}$}.

A connected open set $D \subset {\mathbb{R}}^2$ satisfying the following is {\it a region surrounded by circular segments} or {\it CS-region} of $\mathcal{S}$.
\begin{itemize}
\item The boundary $\overline{D}-D$ is represented as ${\bigcup}_{S_j \in J} C_{\mathcal{S},S_j} \bigcup {S_{\mathcal{S}}}^{\prime}$ and a piecewise smooth curve where the notation for the sets is as follows: each $C_{\mathcal{S},S_j}$ is regarded as a subset of the uniquely defined circle $S_j \in \mathcal{S}$ and a 1-dimensional submanifold (curve) and indexed by an element of some finite subset $J$ of circles from $\mathcal{S}$ and ${S_{\mathcal{S}}}^{\prime}$ is a subset of a secant $S_{\mathcal{S}}$ of $\mathcal{S}$ and a $1$-dimensional submanifold of $S_{\mathcal{S}}$. We call $S_{\mathcal{S}}$ a {\it supporting secant of $(D,\mathcal{S})$}. We call each set $C_{\mathcal{S},S_j}$ a {\it supporting subset of $(D,S_j,\mathcal{S})$} and $J$ a {\it labeling set for $(D,\mathcal{S})$}.
\item The region $D$ contains no points of circles of $\mathcal{S}$.
\end{itemize}
For example, for circles in the plane we have abused the concept that they are sufficiently small with no explicit exposition: we can guess this naturally and we can formulate by using arguments via sufficiently small real number $\epsilon>0$.
A CS-region (of $\mathcal{S}$) of a certain class is {\it sufficiently small} if the area of $\overline{D}$ and the length of $\overline{D}-D$, which are defined of course, are sufficiently small (we can formulate by using arguments via sufficiently small real numbers $\epsilon>0$ for areas of the regions and lengths of the curves). This is seen as a way of natural formulation.


For an NI arrangement $(\mathcal{S},D_{\mathcal{S}})$ of circles, choose a point $x_{j^{\prime},0} \in \overline{D_{\mathcal{S}}}$ contained in a single circle $S_j$ and choose a sufficiently small circle ${S^{\prime}}_{x_{j^{\prime},0}}$ centered there and bounds the closed disk ${D^{\prime}}_{x_{j^{\prime},0}}$. 
We have the unique chord of $\mathcal{S}$ connecting the two points $x_{j^{\prime},0,1}$ and $x_{j^{\prime},0,2}$ of ${S^{\prime}}_{x_{j^{\prime},0}} \bigcap S_j \bigcap \overline{D_{\mathcal{S}}}={S^{\prime}}_{x_{j^{\prime},0}} \bigcap S_j:=\{x_{j^{\prime},0,1}, x_{j^{\prime},0,2}\}$.
We have a chord of $\mathcal{S}$ which is also a subset of $ {D^{\prime}}_{x_{j^{\prime},0}}$ and the two chord boundary points of which are denoted by $x_{j^{\prime},0,3}$ and $x_{j^{\prime},0,4}$, respectively. We can also have one such that the five points $x_{j^{\prime},0,1}$, $x_{j^{\prime},0,3}$, $x_{j^{\prime},0}$, $x_{j^{\prime},0,4}$ and $x_{j^{\prime},0,2}$ are mutually distinct and located on the connected curve ${D^{\prime}}_{x_{j^{\prime},0}} \bigcap S_j$ in this order. We can check the following easily.
\begin{Prop}
\label{prop:1}
For an arbitrary chord of $\mathcal{S}$ obtained above, either it is always a subset of $\overline{D_{\mathcal{S}}} \bigcap {D^{\prime}}_{x_{j^{\prime},0}}$ or a subset of $({\mathbb{R}}^2-D_{\mathcal{S}}) \bigcap {D^{\prime}}_{x_{j^{\prime},0}}$.
\end{Prop}
\begin{Def}
For Proposition \ref{prop:1}, in the former {\rm (}latter{\rm )} case, we call a point $x_{j^{\prime},0}$ a {\it convex} {\rm (}resp. {\it concave}{\rm )} $D_{\mathcal{S}}$-point. We call such a chord of $\mathcal{S}$ a {\it chord at $(D_{\mathcal{S}},x_{j^{\prime},0})$}.
\end{Def}
We can check the following easily.
\begin{Prop}
\label{prop:2}
\begin{enumerate}
\item \label{prop:2.1} Let a point $x_{j^{\prime},0}$ be not a horizontal pole or a vertical pole of $S_j$. For the chord $C_{D_{\mathcal{S}},x_{j^{\prime},0}}$ at $(D_{\mathcal{S}},x_{j^{\prime},0})$ above, we consider the representation of the form $\{p_1+t(p_2-p_1) \mid 0 \leq t \leq 1\}$. For $(p_{1,2,1},p_{1,2,2}):=p_2-p_1$, $p_{1,2,1} \neq 0$ and $p_{1,2,2} \neq 0$ must be satisfied and the signs of $p_{1,2,1}$ and $p_{1,2,2}$ are either of the following for any of these chords at $(D_{\mathcal{S}},x_{j^{\prime},0})$ and these presentations{\rm :} the signs are always same or they are always different and this only depends on the pair $(D_{\mathcal{S}},x_{j^{\prime},0})$.
\item \label{prop:2.2} Let a point $x_{j^{\prime},0}$ be a horizontal pole of $S_j$. For the chord $C_{D_{\mathcal{S}},x_{j^{\prime},0}}$ at $(D_{\mathcal{S}},x_{j^{\prime},0})$ above, we consider the representation of the form $\{p_1+t(p_2-p_1) \mid 0 \leq t \leq 1\}$. For $(p_{1,2,1},p_{1,2,2}):=p_2-p_1$, $p_{1,2,1} \neq 0$ must be satisfied. Furthermore, we can have all cases of $p_{1,2,1}$ and $p_{1,2,2}$ of the following{\rm :} the case $p_{1,2,2} \neq 0$ with the signs being same, the case $p_{1,2,2} \neq 0$ with the signs being different, and the case $p_{1,2,2}=0$.
\item \label{prop:2.3} Let a point $x_{j^{\prime},0}$ be a vertical pole of $S_j$. For the chord $C_{D_{\mathcal{S}},x_{j^{\prime},0}}$ at $(D_{\mathcal{S}},x_{j^{\prime},0})$ above, we consider the representation of the form $\{p_1+t(p_2-p_1) \mid 0 \leq t \leq 1\}$. For $(p_{1,2,1},p_{1,2,2}):=p_2-p_1$, $p_{1,2,2} \neq 0$ must be satisfied. Furthermore, we can have all cases of $p_{1,2,1}$ and $p_{1,2,2}$ of the following{\rm :} the case $p_{1,2,1} \neq 0$ with their signs being same, the case $p_{1,2,1} \neq 0$ with their signs being different, and the case $p_{1,2,1}=0$.
\end{enumerate}
\end{Prop}
We show an example for Proposition \ref{prop:2} (\ref{prop:2.3}) as FIGURE \ref{fig:1}, for example.
\begin{figure}
	\includegraphics[width=80mm,height=80mm]{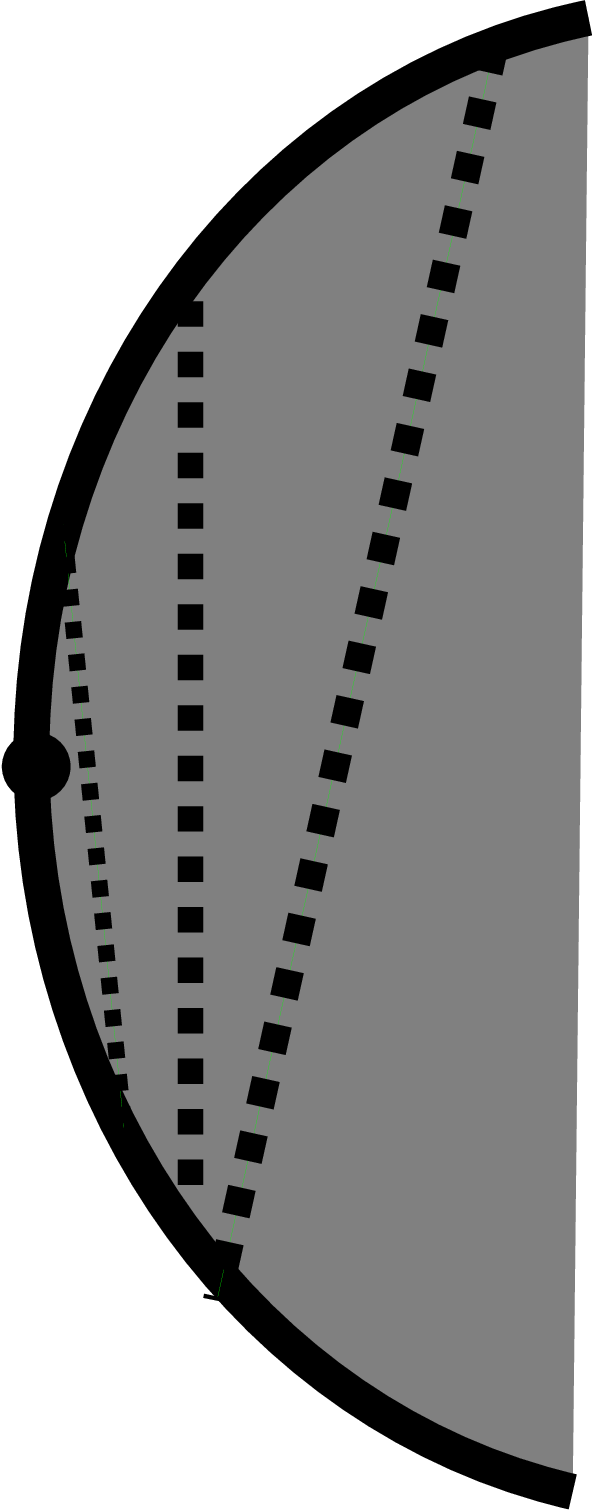}
	\caption{Chords of $\mathcal{S}$ of three types for Proposition \ref{prop:2} (\ref{prop:2.3}).}
	\label{fig:1}
\end{figure}

For any circular secant $S_{\mathcal{S},C_{D_{\mathcal{S}},x_{j^{\prime},0}}}$ of $\mathcal{S}$ passing the two chord boundary points of any chord $C_{D_{\mathcal{S}},x_{j^{\prime},0}}$ at $(D_{\mathcal{S}},x_{j^{\prime},0})$, we have an NCI-MBC arrangement which is also an MBCC arrangement as the pair of $(\{S_{\mathcal{S},C_{D_{\mathcal{S}},x_{j^{\prime},0}}}\},D_{\{S_{\mathcal{S},C_{D_{\mathcal{S}},x_{j^{\prime},0}}}\}})$
where $D_{\{S_{\mathcal{S},C_{D_{\mathcal{S}},x_{j^{\prime},0}}}\}}$ is the uniquely defined bounded region bounded by $S_{\mathcal{S},C_{D_{\mathcal{S}},x_{j^{\prime},0}}}$. 
The circle $S_{\mathcal{S},C_{D_{\mathcal{S}},x_{j^{\prime},0}}}$ is divided into two connected curves diffeomorphic to $D^1$ by the two chord boundary points of $C_{D_{\mathcal{S}},x_{j^{\prime},0}}$ at $(D_{\mathcal{S}},x_{j^{\prime},0})$. 
The chord $C_{D_{\mathcal{S}},x_{j^{\prime},0}}$ at $(D_{\mathcal{S}},x_{j^{\prime},0})$ is contained in the unique straight secant ${S}_{C_{D_{\mathcal{S}},x_{j^{\prime},0}}}$ of $\mathcal{S}$.
For the two resulting connected curves, we can choose the one which and $x_{j^{\prime},0}$ are in the same connected component of ${\mathbb{R}}^2-{S}_{C_{D_{\mathcal{S}},x_{j^{\prime},0}}}$. Let it be denoted by $S_{\mathcal{S},x_{j^{\prime},0}}$ and the remaining curve diffeomorphic to $D^1$ by $S_{\mathcal{S},\overline{x_{j^{\prime},0}}}$.

We can have the unique CS-region of $\{S_{\mathcal{S},C_{D_{\mathcal{S}},x_{j^{\prime},0}}}\}$ depending on $C_{D_{\mathcal{S}},x_{j^{\prime},0}}$
such that for the notation ${\bigcup}_{S_j \in J} C_{\mathcal{S},S_j} \bigcup {S_{\mathcal{S}}}^{\prime}$ {\rm (}${S_{\mathcal{S}}}^{\prime} \subset S_{\mathcal{S}}${\rm )} before $C_{\mathcal{S},S_j}:=S_{\mathcal{S},x_{j^{\prime},0}}$, $J:=\{S_{\mathcal{S},C_{D_{\mathcal{S}},x_{j^{\prime},0}}}\}$, ${S_{\mathcal{S}}}^{\prime}:=C_{D_{\mathcal{S}},x_{j^{\prime},0}}$, and $S_{\mathcal{S}}:={S}_{C_{D_{\mathcal{S}},x_{j^{\prime},0}}}$. Here we can also have the unique CS-region of $\{S_{\mathcal{S},C_{D_{\mathcal{S}},x_{j^{\prime},0}}}\}$ depending on 
$C_{D_{\mathcal{S}},x_{j^{\prime},0}}$ such that for the notation ${\bigcup}_{S_j \in J} C_{\mathcal{S},S_j} \bigcup {S_{\mathcal{S}}}^{\prime}$ before $C_{\mathcal{S},S_j}=S_{\mathcal{S},\overline{x_{j^{\prime},0}}}$, $J=\{S_{\mathcal{S},C_{D_{\mathcal{S}},x_{j^{\prime},0}}}\}$, ${S_{\mathcal{S}}}^{\prime}=C_{D_{\mathcal{S}},x_{j^{\prime},0}}$, and $S_{\mathcal{S}}={S}_{C_{D_{\mathcal{S}},x_{j^{\prime},0}}}$.
For this, Definition \ref{def:7} introduces classes of CS regions of $\{S_{\mathcal{S},C_{D_{\mathcal{S}},x_{j^{\prime},0}}}\}$.
\begin{Def}
\label{def:7}
We call the former resulting CS-region of $\{S_{\mathcal{S},C_{D_{\mathcal{S}},x_{j^{\prime},0}}}\}$ a {\it supported CS-region at $(D_{\mathcal{S}},x_{j^{\prime},0})$}. We call the latter resulting CS-region of $\{S_{\mathcal{S},C_{D_{\mathcal{S}},x_{j^{\prime},0}}}\}$ an {\it unsupported CS region at $(D_{\mathcal{S}},x_{j^{\prime},0})$}.
\end{Def}
\begin{Thm}
\item  \label{thm:1}
\begin{enumerate}
\item  \label{thm:1.1}  Let $x_{j^{\prime},0}$ be a convex $D_{\mathcal{S}}$-point.

If a supported CS-region at $(D_{\mathcal{S}},x_{j^{\prime},0})$ is sufficiently small, then 
we can choose $S_{x_{j^{\prime}},r_{j^{\prime}}}:=S_{\mathcal{S}}$ and $E_{x_{j^{\prime}},r_{j^{\prime}}}:=D_{x_{j^{\prime}},r_{j^{\prime}}}$ in Definition \ref{def:3} to have another NI arrangement ${\mathcal{S}}^{\prime}$ of circles. For an NCI arrangement $(\mathcal{S},D_{\mathcal{S}})$ of circles, $({\mathcal{S}}^{\prime},D_{{\mathcal{S}}^{\prime}})$ is also an NCI arrangement of circles, in this case. For an MBC arrangement $(\mathcal{S},D_{\mathcal{S}})$, $({\mathcal{S}}^{\prime},D_{{\mathcal{S}}^{\prime}})$ is also an MBC arrangement, in this case.

If an unsupported CS-region at $(D_{\mathcal{S}},x_{j^{\prime},0})$ is sufficiently small, then 
we can choose $S_{x_{j^{\prime}},r_{j^{\prime}}}:=S_{\mathcal{S}}$ and $E_{x_{j^{\prime}},r_{j^{\prime}}}:={\mathbb{R}}^2-D_{x_{j^{\prime}},r_{j^{\prime}}}$ in Definition \ref{def:3} to have another NI arrangement ${\mathcal{S}}^{\prime}$ of circles. For an NCI arrangement $(\mathcal{S},D_{\mathcal{S}})$ of circles, $({\mathcal{S}}^{\prime},D_{{\mathcal{S}}^{\prime}})$ is also an NCI arrangement of circles, in this case. For an MBC arrangement $(\mathcal{S},D_{\mathcal{S}})$, $({\mathcal{S}}^{\prime},D_{{\mathcal{S}}^{\prime}})$ is also an MBC arrangement, in this case.
\item  \label{thm:1.2} Let $x_{j^{\prime},0}$ be a concave $D_{\mathcal{S}}$-point.  We restrict the class of supported CS-regions at $(D_{\mathcal{S}},x_{j^{\prime},0})$ to that of CS-regions of $\{S_{\mathcal{S},C_{D_{\mathcal{S}},x_{j^{\prime},0}}}\}$ containing the following uniquely defined bounded CS-region of $\mathcal{S}${\rm :} for the notation ${\bigcup}_{S_j \in J} C_{\mathcal{S},S_j} \bigcup {S_{\mathcal{S}}}^{\prime}$  before $C_{\mathcal{S},S_j}$ is a curve containing $x_{j^{\prime},0}$ with $J=\{S_j\}$, ${S_{\mathcal{S}}}^{\prime}=C_{D_{\mathcal{S}},x_{j^{\prime},0}}$, and $S_{\mathcal{S}}={S}_{C_{D_{\mathcal{S}},x_{j^{\prime},0}}}$. This subclass is not empty.

If a supported CS-region at $(D_{\mathcal{S}},x_{j^{\prime},0})$ of this new class is sufficiently small, then 
we can choose $S_{x_{j^{\prime}},r_{j^{\prime}}}:=S_{\mathcal{S}}={S}_{C_{D_{\mathcal{S}},x_{j^{\prime},0}}}$ and $E_{x_{j^{\prime}},r_{j^{\prime}}}:={\mathbb{R}}^2-D_{x_{j^{\prime}},r_{j^{\prime}}}$ in Definition \ref{def:3} to have another NI arrangement ${\mathcal{S}}^{\prime}$ of circles. For an NCI arrangement $(\mathcal{S},D_{\mathcal{S}})$ of circles, $({\mathcal{S}}^{\prime},D_{{\mathcal{S}}^{\prime}})$ is also an NCI arrangement of circles, in this case. For an MBC arrangement $(\mathcal{S},D_{\mathcal{S}})$, $({\mathcal{S}}^{\prime},D_{{\mathcal{S}}^{\prime}})$ is also an MBC arrangement, in this case.

\end{enumerate}

\end{Thm}
\begin{proof}[Notes on our proof of Theorem \ref{thm:1}]
We can check this by investigating carefully. We only show explicit cases FIGUREs \ref{fig:2}, \ref{fig:3} and \ref{fig:4}.

\begin{figure}
	\includegraphics[width=80mm,height=80mm]{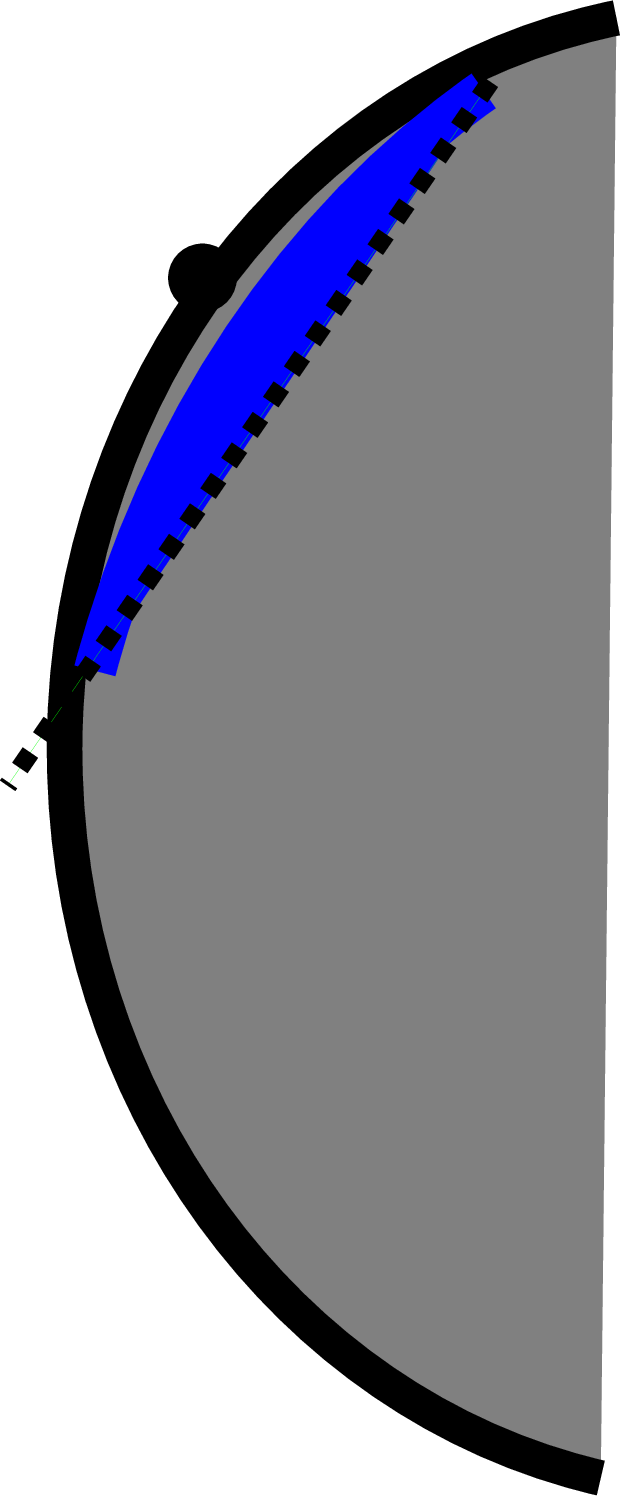}
	\caption{A supported CS-region at $(D_{\mathcal{S}},x_{j^{\prime},0})$ for Theorem \ref{thm:1} (\ref{thm:1.1}) and Theorem \ref{thm:2} (\ref{thm:2.1}).}
	\label{fig:2}
\end{figure}
\begin{figure}
	\includegraphics[width=80mm,height=80mm]{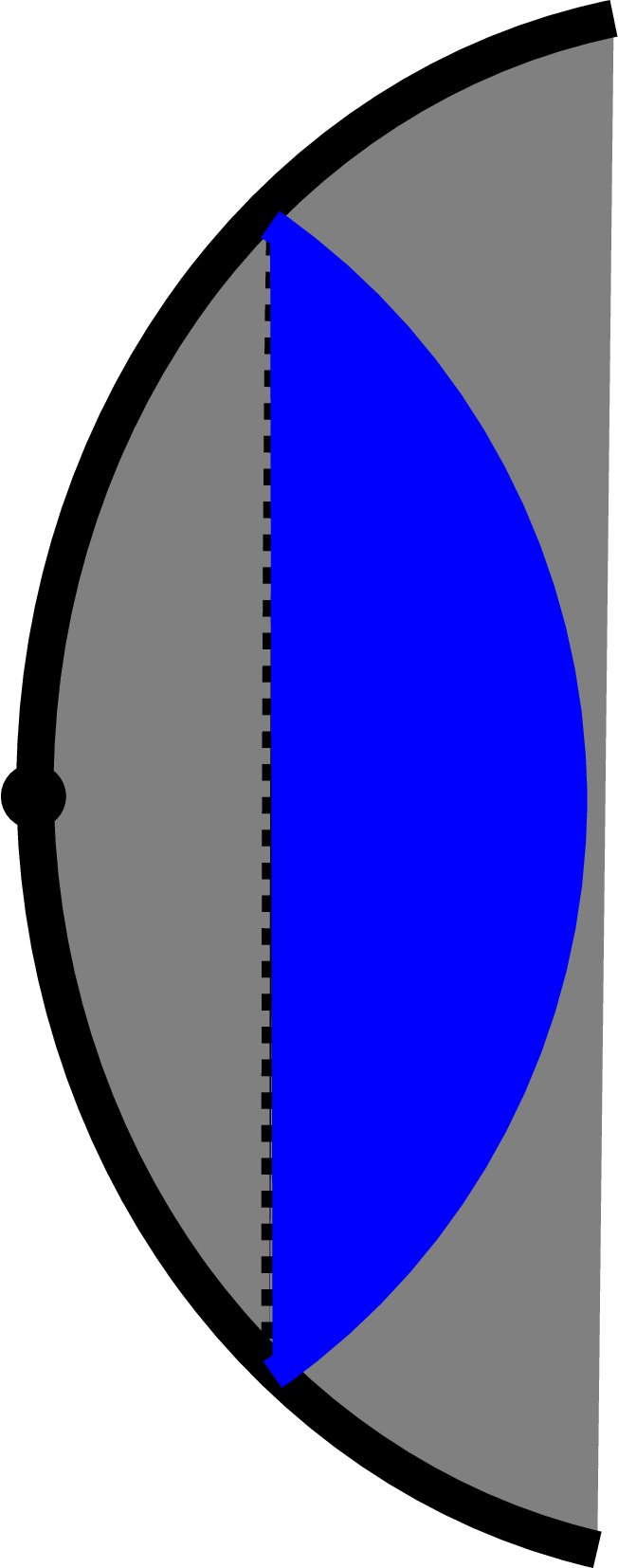}
	\caption{An unsupported CS-region at $(D_{\mathcal{S}},x_{j^{\prime},0})$ for Theorem \ref{thm:1} (\ref{thm:1.1}) and Theorem \ref{thm:2} (\ref{thm:2.2.2}).}
	\label{fig:3}
\end{figure}
\begin{figure}
	\includegraphics[width=80mm,height=80mm]{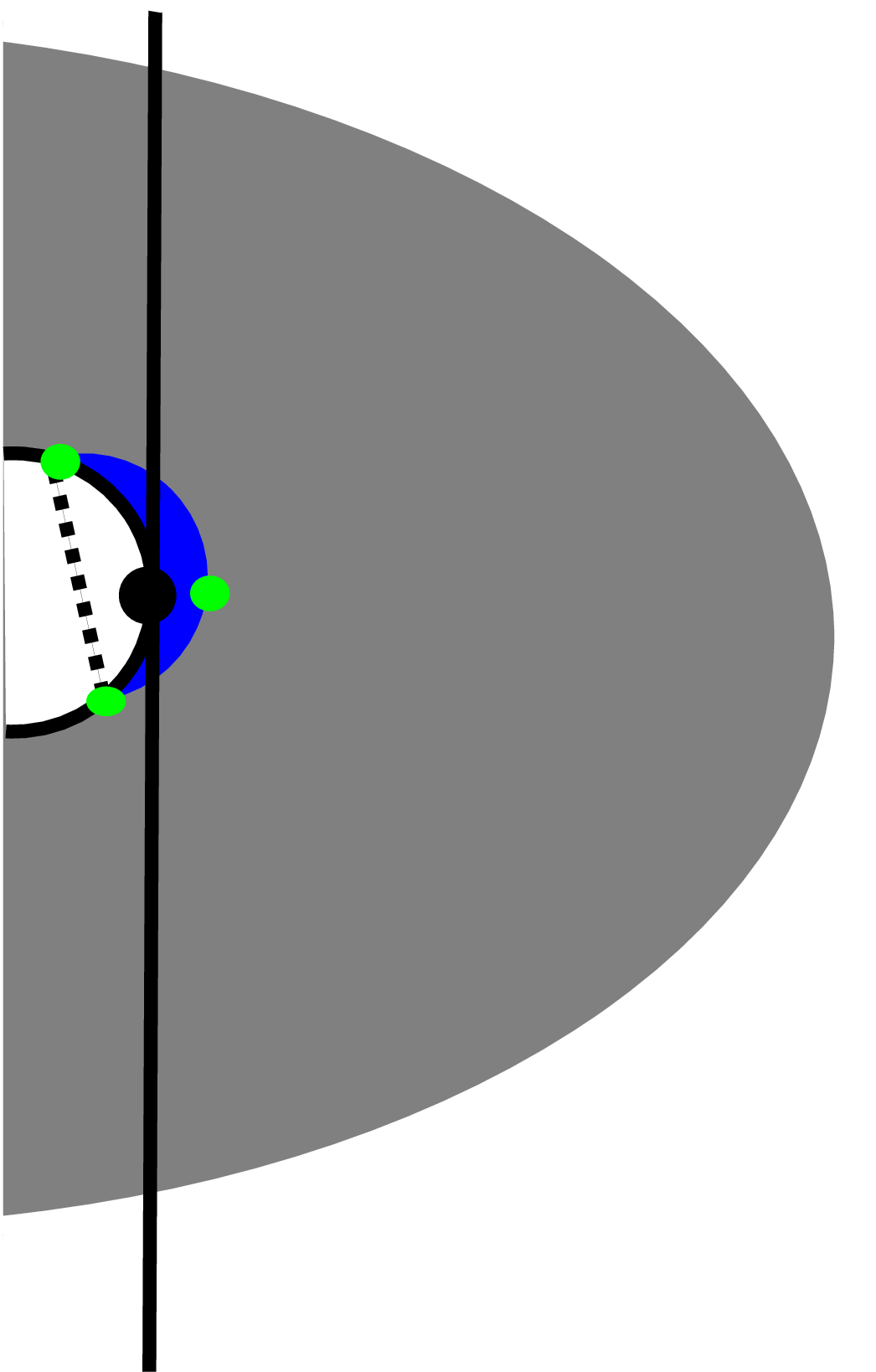}
	\caption{The intersection of a supported CS-region at $(D_{\mathcal{S}},x_{j^{\prime},0})$ and $\overline{D_{\mathcal{S}}}$ for Theorem \ref{prop:2} (\ref{prop:2.2}) and Theorem \ref{thm:5}.}
	\label{fig:4}
\end{figure}
\end{proof}
The following is a variant of \cite[Theorem 2]{kitazawa3} where \cite[Theorem 2]{kitazawa3} is for MBCC arrangements. 
\begin{Thm}
 \label{thm:2}
In Theorem \ref{thm:1}, the local change of the Poincar\'e-Reeb V-digraph of $D_{\mathcal{S}}$ for ${\pi}_{2,1,1}$ to that of $D_{{\mathcal{S}}^{\prime}}$ for ${\pi}_{2,1,1}$ is as follows.
\begin{enumerate}
\item \label{thm:2.1} In the case Proposition \ref{prop:2} {\rm (}\ref{prop:2.1}{\rm )} or {\rm (}\ref{prop:2.2}{\rm )}, the change is explained as either of the following.
\begin{enumerate}
\item \label{thm:2.1.1} For the unique edge $e$ with the interior $e^{\circ} \ni q_{D_{\mathcal{S},1}}(x_{j^{\prime},0})$, two vertices
$v_{x_{j^{\prime},0},1}$ and $v_{x_{j^{\prime},0},2}$ sufficiently close to $q_{D_{\mathcal{S},1}}(x_{j^{\prime},0})$ are added in the interior $e^{\circ}$, in such a way that the relation $V_{D_{\mathcal{S},1}}(v_{x_{j^{\prime},0},1})<q_{D_{\mathcal{S},1}}(x_{j^{\prime},0})<V_{D_{\mathcal{S},1}}(v_{x_{j^{\prime},0},2})$ holds.
\item \label{thm:2.1.2} For the unique vertex $q_{D_{\mathcal{S},1}}(x_{j^{\prime},0})$ and the unique pair $(e_{x_{j^{\prime},0},1},e_{x_{j^{\prime},0},2})$ of oriented edges, we add two vertices
$v_{x_{j^{\prime},0},1} \in {e_{x_{j^{\prime},0},1}}^{\circ}$ and $v_{x_{j^{\prime},0},2} \in e_{x_{j^{\prime},0},2}^{\circ}$ sufficiently close to $q_{D_{\mathcal{S},1}}(x_{j^{\prime},0})$ in their interiors{\rm :} originally the edge $e_{x_{j^{\prime},0},1}$ enters the vertex $q_{D_{\mathcal{S},1}}(x_{j^{\prime},0})$ and the edge $e_{x_{j^{\prime},0},2}$ departs from the vertex $q_{D_{\mathcal{S},1}}(x_{j^{\prime},0})$. 
\end{enumerate}
\item \label{thm:2.2} In the case Proposition \ref{prop:2} {\rm (}\ref{prop:2.3}{\rm )} and Theorem \ref{thm:1} {\rm (}\ref{thm:1.1}{\rm )}, ${\pi}_{2,1,1}(x_{j^{\prime},0})$ is the minimum or the maximum in the image of the function ${\pi}_{2,1,1} {\mid}_{\overline{D_{\mathcal{S}}}}$ and $x_{j^{\prime},0}$ is the unique point such that the value is the minimum or the maximum.
Let ${\pi}_{2,1,1}(x_{j^{\prime},0})$ be the minimum in the image. The change is explained as either of the following and both of these two can be realized.
\begin{enumerate}
\item \label{thm:2.2.1} This is for the case of considering a sufficiently small CS-region at $(D_{\mathcal{S}},x_{j^{\prime},0})$ in Theorem \ref{thm:1} {\rm (}\ref{thm:1.1}{\rm )} except an unsupported CS-region $D^{\prime}$ at $(D_{\mathcal{S}},x_{j^{\prime},0})$ with a supporting secant of $(D^{\prime},\mathcal{S})$ being defined for the case $p_{1,2,1}=0$ in Proposition \ref{prop:2} {\rm (}\ref{prop:2.3}{\rm )} and as the unique extension of the chord of $\mathcal{S}$. 

The unique edge $e$ containing $q_{D_{\mathcal{S},1}}(x_{j^{\prime},0})$ as a vertex is changed into two adjacent oriented edges $e_1$ and $e_2$ as follows.

Let $v_e$ be the remaining vertex contained in $e$. Let $e_2$ be a new edge departing from a new vertex ${v_e}^{\prime}$ and entering $v_e$. Let $e_1$ be a new edge departing from a new vertex ${v_{e,0}}$ and entering ${v_e}^{\prime}$. In addition, respecting this structure of the digraph, we define the structure of the new V-digraph $G_{D_{{\mathcal{S}}^{\prime},1}}$ as the digraph associated with a new function $l_{G_{D_{{\mathcal{S}}^{\prime},1}}}:=V_{D_{{\mathcal{S}}^{\prime}},1}$ on the vertex set whose values at  ${v_{e,0}}$ and ${v_e}^{\prime}$ are sufficiently close to ${\pi}_{2,1,1}(x_{j^{\prime},0})$ and whose values at the remaining vertices are same as those of the {\rm (}originally associated{\rm )} function $l_{G_{D_{\mathcal{S},1}}}:=V_{D_{{\mathcal{S}}},1}$.
\item \label{thm:2.2.2} This is for the case of considering a sufficiently small unsupported CS-region at $(D_{\mathcal{S}},x_{j^{\prime},0})$ in Theorem \ref{thm:1} {\rm (}\ref{thm:1.1}{\rm )} with a supporting secant of $(D^{\prime},\mathcal{S})$ being defined for the case $p_{1,2,1}=0$ in Proposition \ref{prop:2} {\rm (}\ref{prop:2.3}{\rm )} and as the unique extension of the chord of $\mathcal{S}$. 

The unique edge $e$ containing $q_{D_{\mathcal{S},1}}(x_{j^{\prime},0})$ as a vertex is changed into three mutually adjacent oriented edges $e_{1,1}$, $e_{1,2}$ and $e_{2}$ as follows.

Let $v_e$ be the remaining vertex contained in $e$. Let $e_2$ be a new edge departing from a new vertex ${v_e}^{\prime}$ and entering $v_e$. Let $e_{1,i^{\prime}}$ {\rm (}$i^{\prime}=1,2${\rm )} be another new edge departing from another new vertex $v_{e,i^{\prime},0}$ and entering ${v_e}^{\prime}$ {\rm (}$v_{e,1,0} \neq v_{e,2,0}${\rm )}.  In addition, respecting this structure of the digraph, we define the structure of the new V-digraph $G_{D_{{\mathcal{S}}^{\prime},1}}$ as the digraph associated with a new function $l_{G_{D_{{\mathcal{S}}^{\prime},1}}}:=V_{D_{{\mathcal{S}}^{\prime}},1}$ on the vertex set whose values at ${v_{e,i^{\prime},0}}$ and ${v_e}^{\prime}$ are sufficiently close to ${\pi}_{2,1,1}(x_{j^{\prime},0})$, whose values at ${v_{e,1,0}}$ and ${v_{e,2,0}}$ are same, and whose values at the remaining vertices are same as those of the {\rm (}originally associated{\rm )} function $l_{G_{D_{\mathcal{S},1}}}:=V_{D_{{\mathcal{S}}},1}$.
\end{enumerate}
In the case ${\pi}_{2,1,1}(x_{j^{\prime},0})$ is the maximum in the image of the function ${\pi}_{2,1,1} {\mid}_{\overline{D_{\mathcal{S}}}}$, we have a similar corresponding fact.
\end{enumerate}
\end{Thm}
\begin{proof}[Notes on our proof of Theorem \ref{thm:2}]
We can prove Theorem \ref{thm:2} by investigating carefully one by one. This is also important in \cite{kitazawa3}.
\end{proof}

For an NI arrangement $(\mathcal{S},D_{\mathcal{S}})$ of circles, choose a point $x_{j^{\prime},0}  \in \overline{D_{\mathcal{S}}}$ contained in two distinct circles $S_{j_1},S_{j_2} \in \mathcal{S}$ and choose a sufficiently small circle ${S^{\prime}}_{x_{j^{\prime},0}}$ centered there and bounds the closed disk ${D^{\prime}}_{x_{j^{\prime},0}}$.
We have the unique chord of $\mathcal{S}$ connecting the two points $x_{j^{\prime},0,1}$ and $x_{j^{\prime},0,2}$ of ${S^{\prime}}_{x_{j^{\prime},0}} \bigcap \overline{D_{\mathcal{S}}} \bigcap (S_{j_1} \bigcup S_{j_2}):=\{x_{j^{\prime},0,1}, x_{j^{\prime},0,2}\}$.
We have a chord of $\mathcal{S}$ which is also a subset of $ {D^{\prime}}_{x_{j^{\prime},0}}$ and the two chord boundary points of which are denoted by $x_{j^{\prime},0,3}$ and $x_{j^{\prime},0,4}$, respectively. We also have one such that the five points $x_{j^{\prime},0,1}$, $x_{j^{\prime},0,3}$, $x_{j^{\prime},0}$, $x_{j^{\prime},0,4}$ and $x_{j^{\prime},0,2}$ are mutually distinct and located on the connected curve ${D^{\prime}}_{x_{j^{\prime},0}} \bigcap (S_{j_1} \bigcup S_{j_2})$ in this order.  
\begin{Prop}
\label{prop:3}
For an arbitrary chord of $\mathcal{S}$ above, it is always a subset of $\overline{D_{\mathcal{S}}} \bigcap {D^{\prime}}_{x_{j^{\prime},0}}$. We call such a chord of $\mathcal{S}$ a {\rm chord at $(D_{\mathcal{S}},x_{j^{\prime},0})$}.
\end{Prop}

For any circular secant $S_{\mathcal{S}}$ of $\mathcal{S}$ passing the two chord boundary points of any chord $C_{D_{\mathcal{S}},x_{j^{\prime},0}}$ at $(D_{\mathcal{S}},x_{j^{\prime},0})$, we have a CS-region $D$ of $\mathcal{S}$ such that 
$x_{j^{\prime},0}$ is contained in the closure $\overline{D}$, that $S_{\mathcal{S}}$ is a supporting secant of $(D,\mathcal{S})$ and that a labeling set for $(D,\mathcal{S})$ is a two-element set $\{S_{j_1},S_{j_2}\}$. For a point $(p_{j_1,1},p_{j_1,2}) \neq x_{j^{\prime},0}$ in a supporting subset $C_{\mathcal{S},S_{j_1}}$ of $(D,S_{j_1},\mathcal{S})$ and a point $(p_{j_2,1},p_{j_2,2}) \neq x_{j^{\prime},0}$ in a supporting subset $C_{\mathcal{S},S_{j_2}}$ of $(D,S_{j_2},\mathcal{S})$, we have
 a vector $((p_{j_1,1},p_{j_1,2})-x_{j^{\prime},0}, (p_{j_2,1},p_{j_2,2})-x_{j^{\prime},0}) \in {\mathbb{R}}^4$. Then we uniquely have the corresponding tuple of the signs of these components, whose values are not $0$, depending on $(\mathcal{S},D_{\mathcal{S}})$, the point $x_{j^{\prime},0}$ and the ordered pair $(S_{j_1},S_{j_2})$ only, and independent of a supporting secant $S_{\mathcal{S}}$ of $(D,\mathcal{S})$, and a supporting subset $C_{\mathcal{S},S_{j_i}}$ of $(D,S_{j_i},\mathcal{S})$ {\rm (}$i=1,2${\rm )}, for example. 

\begin{Def}
We call the tuple of the $4$ signs above {\it the signs at $(D_{\mathcal{S}},x_{j^{\prime},0})$}.
\end{Def}
\begin{Thm}
\label{thm:3}
\begin{enumerate}
\item
\label{thm:3.1}
\begin{enumerate}
\item \label{thm:3.1.1} Let $x_{j^{\prime},0}$ be a point such that the sign of the 1st component and that of the 3rd component of the signs at $(D_{\mathcal{S}},x_{j^{\prime},0})$ are same and that the sign of the 2nd component and that of the 4th component of the signs at $(D_{\mathcal{S}},x_{j^{\prime},0})$ are same. Then we can have all cases of chords at $(D_{\mathcal{S}},x_{j^{\prime},0})$ as in Proposition \ref{prop:2} {\rm (}\ref{prop:2.2}{\rm )} and {\rm (}\ref{prop:2.3}{\rm )}.
\item \label{thm:3.1.2} Let $x_{j^{\prime},0}$ be a point such that the sign of the 1st component and that of the 3rd component of the signs at $(D_{\mathcal{S}},x_{j^{\prime},0})$ are same and that the sign of the 2nd component and that of the 4th component of the signs at $(D_{\mathcal{S}},x_{j^{\prime},0})$ are mutually distinct. Then 
we can have all cases of chords at $(D_{\mathcal{S}},x_{j^{\prime},0})$ as in Proposition \ref{prop:2} {\rm (}\ref{prop:2.3}{\rm )} and we cannot have chords at $(D_{\mathcal{S}},x_{j^{\prime},0})$ as in Proposition \ref{prop:2} {\rm (}\ref{prop:2.2}{\rm )} with $p_{1,2,2}=0$.

\item \label{thm:3.1.3} Let $x_{j^{\prime},0}$ be a point such that the sign of the 1st component and that of the 3rd component of the signs at $(D_{\mathcal{S}},x_{j^{\prime},0})$ are distinct and that the sign of the 2nd component and that of the 4th component of the signs at $(D_{\mathcal{S}},x_{j^{\prime},0})$ are same. Then we can have all cases of chords at $(D_{\mathcal{S}},x_{j^{\prime},0})$ as in Proposition \ref{prop:2} {\rm (}\ref{prop:2.2}{\rm )} and we cannot have chords at $(D_{\mathcal{S}},x_{j^{\prime},0})$ as in Proposition \ref{prop:2} {\rm (}\ref{prop:2.3}{\rm )} with $p_{1,2,1}=0$.
\item \label{thm:3.1.4} Let $x_{j^{\prime},0}$ be a point such that the sign of the 1st component and that of the 3rd component of the signs at $(D_{\mathcal{S}},x_{j^{\prime},0})$ are distinct and that the sign of the 2nd component and that of the 4th component of the signs at $(D_{\mathcal{S}},x_{j^{\prime},0})$ are distinct. Then we can have a case of chords at $(D_{\mathcal{S}},x_{j^{\prime},0})$ as in Proposition \ref{prop:2} {\rm (}\ref{prop:2.1}{\rm )} and this is the unique case.
\end{enumerate}
\item \label{thm:3.2} Let $x_{j^{\prime},0}$ be a point such that for some straight line $L_{x_{j^{\prime},0}}$, the intersection satisfies $\overline{D_{\mathcal{S}}} \bigcap L_{x_{j^{\prime},0}}=\{x_{j^{\prime},0}\}$. Then we can find chords at $(D_{\mathcal{S}},x_{j^{\prime},0})$ to do a discussion similar to that in Theorem \ref{thm:1} {\rm (}\ref{thm:1.1}{\rm )}. 
\end{enumerate}
%
\end{Thm}
\begin{proof}[Notes on our proof of Theorem \ref{thm:3}]
We give notes on our proof of Theorem \ref{thm:3}. As the previous theorems, we can prove Theorem \ref{thm:3} by investigating carefully one by one. We only present FIGUREs \ref{fig:5} and \ref{fig:6} for some cases of chords at $(D_{\mathcal{S}},x_{j^{\prime},0})$, for example.

\end{proof}
The following is a variant of \cite[Theorem 3]{kitazawa3} where \cite[Theorem 3]{kitazawa3} is for MBCC arrangements. This is also regarded as a variant of Theorem \ref{thm:2} and we abuse the notions, notation and rules.
\begin{figure}
	\includegraphics[width=40mm,height=40mm]{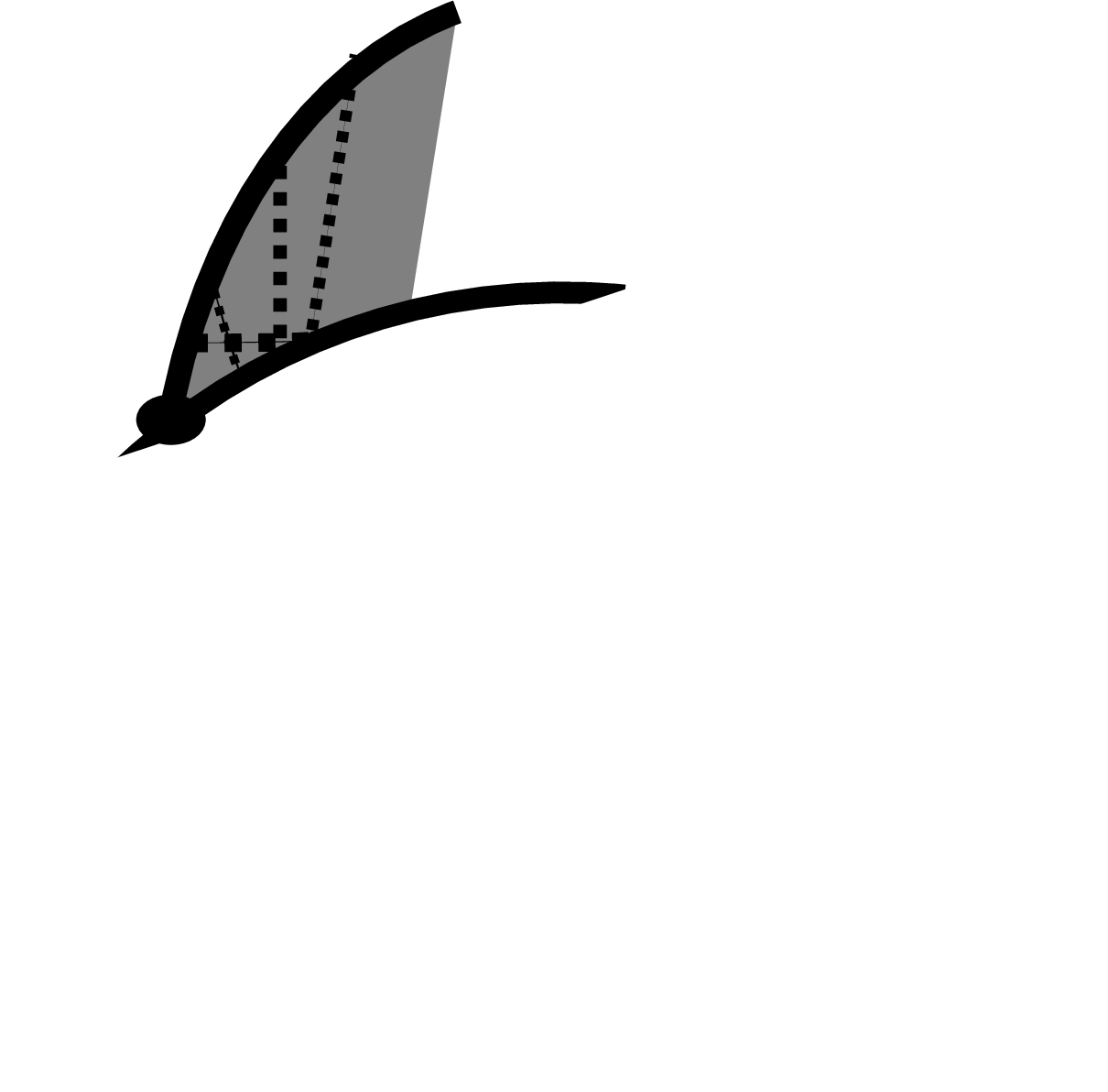}
	\caption{Four types of chords at $(D_{\mathcal{S}},x_{j^{\prime},0})$ for Theorem \ref{thm:3} (\ref{thm:3.1.1}).}
	\label{fig:5}
\end{figure}
\begin{figure}
	\includegraphics[width=40mm,height=40mm]{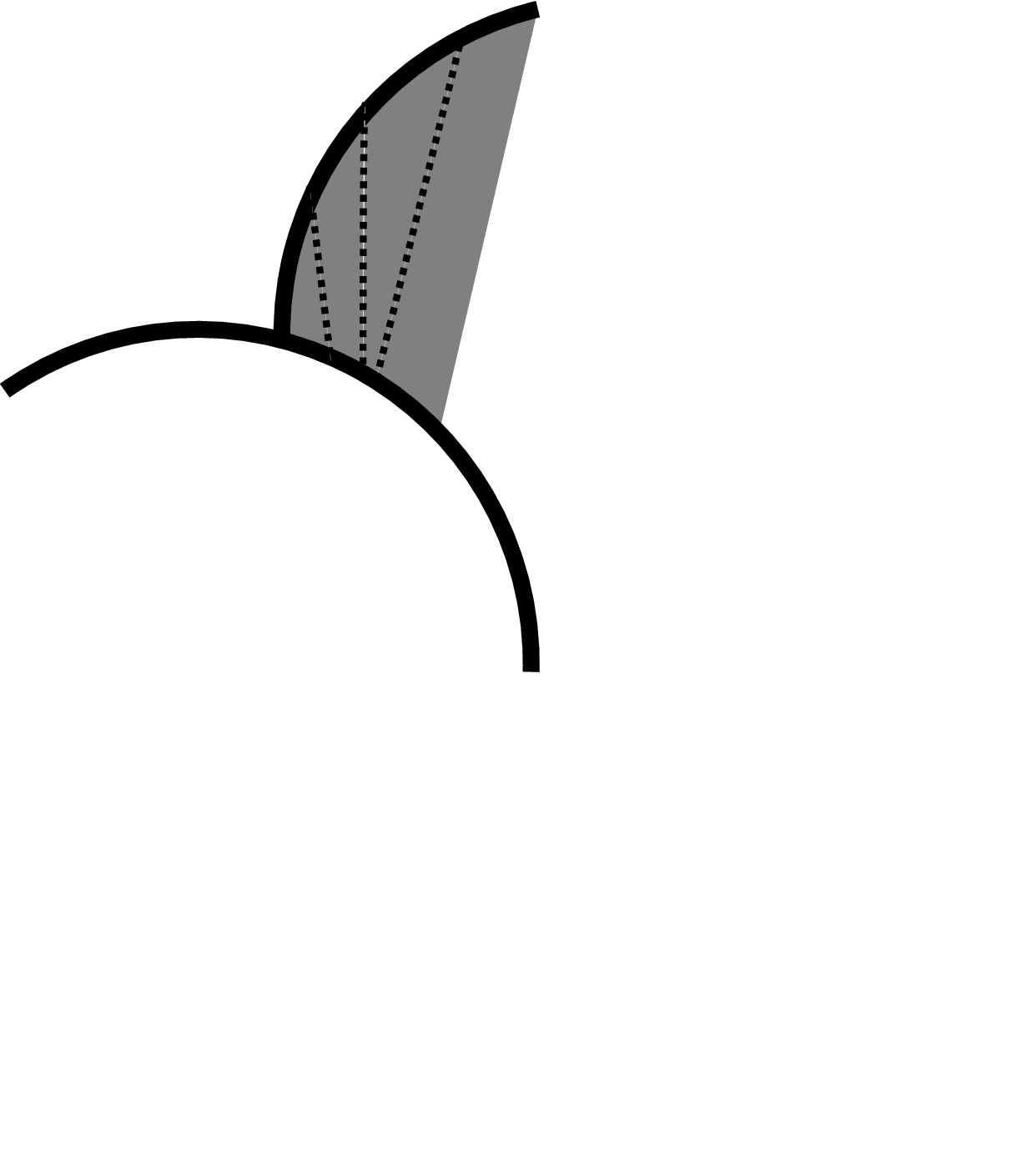}
	\caption{Three types of chords at $(D_{\mathcal{S}},x_{j^{\prime},0})$ for Theorem \ref{thm:3} (\ref{thm:3.1.2}).}
	\label{fig:6}
\end{figure}

\begin{Thm}
\label{thm:4}The local change of the Poincar\'e-Reeb V-digraph of $D_{\mathcal{S}}$ for ${\pi}_{2,1,1}$ to that of $D_{{\mathcal{S}}^{\prime}}$ for ${\pi}_{2,1,1}$ in the situation of  Theorem \ref{thm:3} is as follows.
\begin{enumerate}
	\item \label{thm:4.1}
	In the situation of Theorem \ref{thm:3} {\rm (}\ref{thm:3.1.1}{\rm )} or {\rm (}\ref{thm:3.1.2}{\rm )} with Theorem \ref{thm:3} {\rm (}\ref{thm:3.2}{\rm )}, we can argue as Theorem \ref{thm:2} {\rm (}\ref{thm:2.2}{\rm )}.
	\item \label{thm:4.2} In the situation of Theorem \ref{thm:3} {\rm (}\ref{thm:3.1.3}{\rm )} or {\rm (}\ref{thm:3.1.4}{\rm )} with Theorem \ref{thm:3} {\rm (}\ref{thm:3.2}{\rm )}, we can argue as Theorem \ref{thm:2} {\rm (}\ref{thm:2.1}{\rm )}.
\end{enumerate}
\end{Thm}
\begin{proof}[Notes on our proof of Theorem \ref{thm:4}]
	We give notes on our proof of Theorem \ref{thm:4}.

	The case (\ref{thm:4.1}) is shown by investigating carefully as ever. We only note that in the case the chord at $(D_{\mathcal{S}},x_{j^{\prime},0})$ is as in Proposition \ref{prop:2} {\rm (}\ref{prop:2.2}{\rm )} with $p_{1,2,2}=0$, we can argue as Proposition \ref{prop:2} (\ref{prop:2.1}). We have not encountered cases similar to this case and this must be considered in the case Theorem \ref{thm:3} (\ref{thm:3.1.1}) with (\ref{thm:3.2}). See also FIGURE \ref{fig:5}.  
	
		The case (\ref{thm:4.2}) is easily shown by investigating carefully as ever.
	\end{proof}
The following is a variant of \cite[Theorem 5]{kitazawa3} where \cite[Theorem 5]{kitazawa3} is for MBCC arrangements. 
\begin{Thm}
\label{thm:5}
The local change of the Poincar\'e-Reeb V-digraph of $D_{\mathcal{S}}$ for ${\pi}_{2,1,1}$ to that of $D_{{\mathcal{S}}^{\prime}}$ for ${\pi}_{2,1,1}$ in the situation of Proposition \ref{prop:2} {\rm (}\ref{prop:2.3}{\rm )} with Theorem \ref{thm:1} {\rm (}\ref{thm:1.2}{\rm )}  is same as the case of \cite[Theorem 5 with Problem 1]{kitazawa3} where for the numbers "$i(v_1)=i(v_2)$ there", sufficiently close to ${\pi}_{2,1,1}(x_{j^{\prime},0})$, the relation $i(v_1)>i(v_2)$ or $i(v_1)<i(v_2)$ may hold instead. We can have all cases for the relation between the numbers $i(v_1)$ and $i(v_2)$, which are sufficiently close to ${\pi}_{2,1,1}(x_{j^{\prime},0})$ and both of which are either smaller or larger than ${\pi}_{2,1,1}(x_{j^{\prime},0})$.
\end{Thm}
\begin{proof}[Notes on our proof of Theorem \ref{thm:5}]
For the graphs we need, we can understand the definition by checking \cite{kitazawa3}. More precisely, we only need to check around \cite[Theorem 5 with Problem 1]{kitazawa3} where we do not assume non-trivial mathematical arguments in the definition and formulation. 

However, we give our related exposition essentially same as the original one.

Let $x_{j^{\prime},0}$ be a vertical pole where the restriction of ${\pi}_{2,1}$ to the boundary of $\overline{D_{\mathcal{S}}}-D_{\mathcal{S}}$ has a local maximum.

There exist exactly two edges $e_{x_{j^{\prime},0},1,-}$ and $e_{x_{j^{\prime},0},2,-}$ which enter $q_{D_{\mathcal{S}},1}(x_{j^{\prime},0})$ and the preimages of the closures of which for the quotient map contain $x_{j^{\prime},0}$.
There exist exactly one edge $e_{x_{j^{\prime},0},+}$ which departs from $q_{D_{\mathcal{S}},1}(x_{j^{\prime},0})$ and the preimage of the closure of which for the quotient map contains $x_{j^{\prime},0}$.
For the remaining edges entering or departing from $q_{D_{\mathcal{S}},1}(x_{j^{\prime},0})$, we change the vertex they enter or depart from.
For each of the two edges $e_{x_{j^{\prime},0},i,-}$ ($i=1,2$), we add two adjacent vertices $v_{x_{j^{\prime},0},i,-,0}$ and $v_{x_{j^{\prime},0},i,-}$ in the interior and sufficiently close to $q_{D_{\mathcal{S}},1}(x_{j^{\prime},0})$. We also define the vertices in such a way that $v_{x_{j^{\prime},0},i,-,0}$ is closer to $q_{D_{\mathcal{S}},1}(x_{j^{\prime},0})$ than $v_{x_{j^{\prime},0},i,-}$. For the remaining edges entering $q_{D_{\mathcal{S}},1}(x_{j^{\prime},0})$, we change the vertex they enter to $v_{x_{j^{\prime},0},1,-,0}$ or $v_{x_{j^{\prime},0},2,-,0}$, according to the location of the preimages for the quotient map, in ${\mathbb{R}}^2$.
For the edges departing from $q_{D_{\mathcal{S}},1}(x_{j^{\prime},0})$ except $e_{x_{j^{\prime},0},+}$, we change the vertex they depart from to $v_{x_{j^{\prime},0},1,-,0}$ or $v_{x_{j^{\prime},0},2,-,0}$, according to the location of the preimages for the quotient map, in ${\mathbb{R}}^2$.

In addition, respecting the structures of our digraphs, we define the structure of the new V-digraph $G_{D_{{\mathcal{S}}^{\prime},1}}$ as the digraph associated with a new function $l_{G_{D_{{\mathcal{S}}^{\prime},1}}}:=V_{D_{{\mathcal{S}}^{\prime}},1}$ on the vertex set whose values at $v_{x_{j^{\prime},0},i,-,0}$ are ${\pi}_{2,1,1}(x_{j^{\prime},0})$, whose values at $v_{x_{j^{\prime},0},i,-}$ are $i(v_i)<{\pi}_{2,1,1}(x_{j^{\prime},0})$, whose value at $q_{D_{\mathcal{S}},1}(x_{j^{\prime},0})$ is changed into a real number greater than  ${\pi}_{2,1,1}(x_{j^{\prime},0})$ and sufficiently close to it, and whose values at the remaining vertices are same as those of the {\rm (}originally associated{\rm )} function $l_{G_{D_{\mathcal{S},1}}}:=V_{D_{{\mathcal{S}}},1}$.

FIGURE \ref{fig:4} shows an example for the case $i(v_1)>i(v_2)$: we can have the cases $i(v_1)<i(v_2)$ and $i(v_1)=i(v_2)$ by considering chords at $(D_{\mathcal{S}},x_{j^{\prime},0})$, according to Proposition \ref{prop:2} (\ref{prop:2.3}) with Theorem \ref{thm:1} (\ref{thm:1.2}).

Let $x_{j^{\prime},0}$ be a vertical pole where the restriction of ${\pi}_{2,1}$ to the boundary of $\overline{D_{\mathcal{S}}}-D_{\mathcal{S}}$
has a local minimum. In this case, we can argue in the same way.

This completes our exposition on the proof.
\end{proof}
For Theorems \ref{thm:1}--\ref{thm:5}, we have another similar result for the Poincar\'e-Reeb V-digraph of $D_{\mathcal{S}}$ for ${\pi}_{2,1,2}$ and that of $D_{{\mathcal{S}}^{\prime}}$ for ${\pi}_{2,1,2}$ by a kind of symmetries. This is also discussed in \cite{kitazawa3}.
\begin{Def}
\label{def:6}
In Definition \ref{def:3} and in the step (\ref{def:3.2}) there, 
we consider the operations in Theorems \ref{thm:1} and \ref{thm:3}. This defines an {\it NI arrangement supported by small chords} ({\it SSC-NI arrangement}).
\end{Def}
The following is a short corollary and we can check this by our construction.
\begin{Cor}
An SSC-NI arrangement is also an NCI arrangement of circles and MBC arrangement.
\end{Cor}

\begin{Ex}
\label{ex:1}
\cite[Main Theorems 4 and 5 and FIGURE 8]{kitazawa2} has mainly motivated us to introduce SSC-NI arrangements and \cite[Example 2 and FIGURE 8]{kitazawa3} shows a simplest case of SSC-NI arrangements.
\end{Ex}
\begin{Ex}
\label{ex:2}
	Another simplest example $(\mathcal{S},D_{\mathcal{S}})$ of SSC-NI is shown in FIGURE \ref{fig:7}. The set $\mathcal{S}$ consists of exactly $3$ circles where the initial set ${\mathcal{S}}_0$ of circles, defined in Definition \ref{def:3}, consists of just one circle: the circle in the left is added according to Proposition \ref{prop:2} (\ref{prop:2.1}) and after that another circle, depicted partially by an arc, is added according to Theorem \ref{thm:3} (\ref{thm:3.1.4}) with (\ref{thm:3.2}).
\begin{figure}
	\includegraphics[width=80mm,height=80mm]{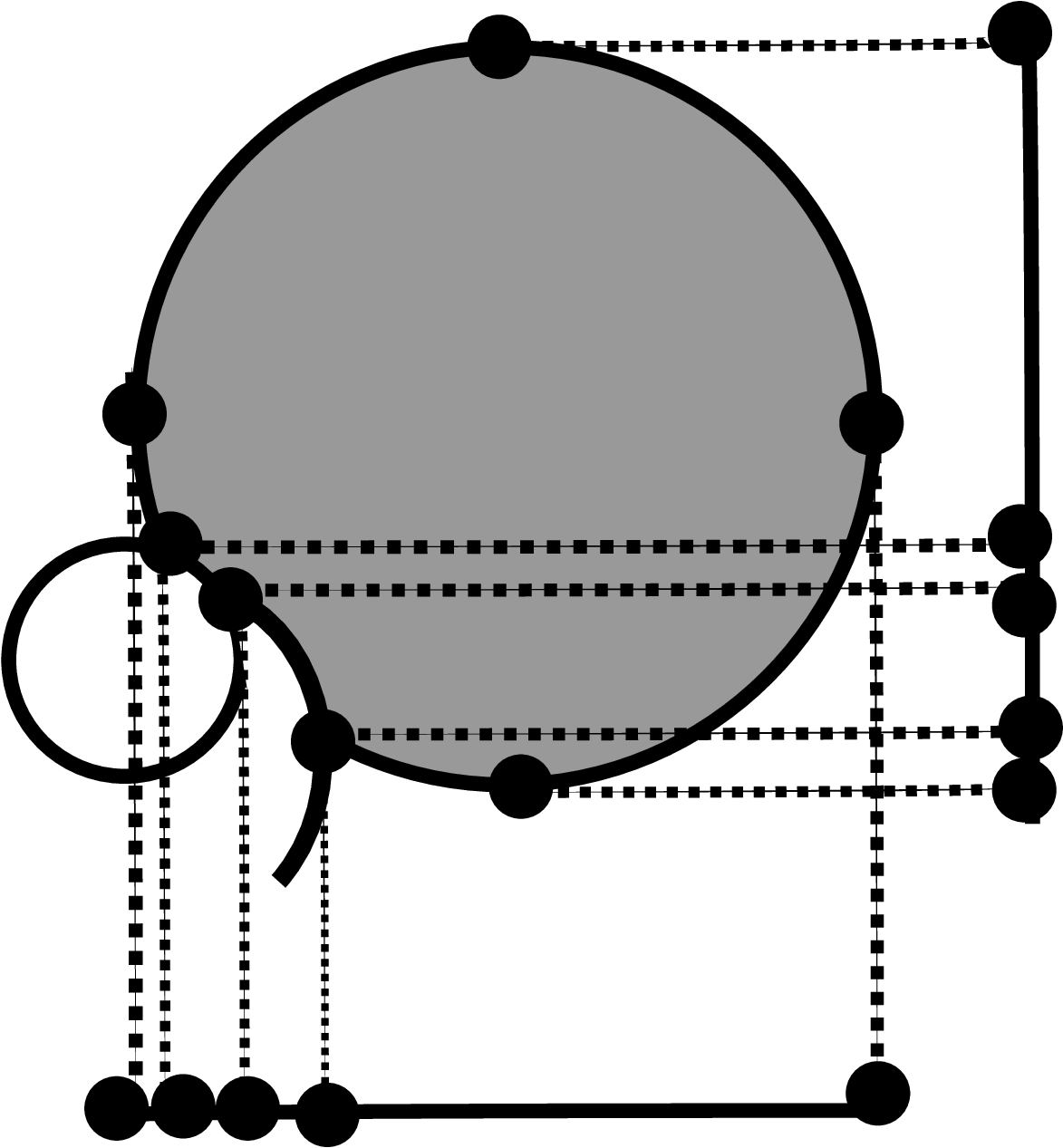}
	\caption{Example \ref{ex:2}: the set $\mathcal{S}$ consists of exactly $3$ circles where the initial set ${\mathcal{S}}_0$ consists of one circle and two circles are added one after another.}
	\label{fig:7}
\end{figure}

The Poincar\'e-Reeb V-digraph of $D_{\mathcal{S}}$ for ${\pi}_{2,1,i}$ is isomorphic to a graph with exactly five vertices and homeomorphic to the interval $D^1$ for $i=1,2$.
If we check \cite{kitazawa3}, especially, \cite[Theorems 2 and 3]{kitazawa3}, then we can see that the pair of Poincar\'e-Reeb graphs in this case is not realized by considering MBCC arrangements only.
\end{Ex}
We close our paper by presenting related problems. Similar problems are also presented in \cite{kitazawa3}.
\begin{Prob}
Find and formulate other important explicit classes of NI arrangements of circles. Investigate the classes from the viewpoint of combinatorics, singularity theory of smooth functions and maps, and real algebraic geometry, for example.
\end{Prob}
The following is an explicit related problem. 
\begin{Prob}[\cite{kitazawa1, kitazawa2}]
For a given finite V-digraph, can we reconstruct nice and explicit real algebraic functions whose Reeb V-digraphs are isomorphic to this.
\end{Prob}
This is also presented in \cite{kitazawa3}. Originally, this is regarded as a problem in the differentiable (smooth) category (\cite{sharko}). For related studies in the differentiable (smooth) category, check \cite{saeki1} for example and check also \cite{kitazawa3} again.
\section{Conflict of interest and Data availability.}
\noindent {\bf Conflict of interest.} \\
The author works at Institute of Mathematics for Industry (https://www.jgmi.kyushu-u.ac.jp/en/about/young-mentors/) and this is closely related to our study. Our study thanks them for the supports. The author is also a researcher at Osaka Central
Advanced Mathematical Institute (OCAMI researcher), supported by MEXT Promotion of Distinctive Joint Research Center Program JPMXP0723833165. Although he is not employed there, our study also thanks this. \\
\ \\
{\bf Data availability.} \\
Data essentially supporting our present study are all in the present
 paper.

\end{document}